\documentclass[12pt,reqno,oneside]{amsart}
      
\usepackage{amsmath,amsthm,amsfonts,amssymb}
\usepackage{indentfirst}
\usepackage{url}
\usepackage{graphicx}
\usepackage[usenames]{color}
\usepackage{ulem}

\theoremstyle{plain}
\newtheorem{theorem}{Theorem}[section]

\newtheorem{lem}[theorem]{Lemma}
\newtheorem{prop}[theorem]{Proposition}

\theoremstyle{definition}

\newtheorem{exa}[theorem]{Example}
\newtheorem{obs}[theorem]{Remark}

\numberwithin{equation}{section}

\newcommand{\mc}{C(\lambda, p)}

\newcommand{\MCsr}{C_*(\lambda,p)}

\newcommand{\MCI}{C_d(\lambda,p)}
\newcommand{\MCId}{C_2(\lambda,p)}
\newcommand{\MCIt}{C_3(\lambda,p)}

\begin{document}

\baselineskip=18pt

\title[Extinction time in growth models subject to binomial catastrophes]{Extinction 
	time in growth models subject to binomial catastrophes}

\author[Frank Duque]{F. Duque}
\address[Frank Duque]{Escuela de Matem\'aticas, Universidad Nacional de Colombia, Calle 59A no 63-20, Medellin, Colombia}
\email{frduquep@unal.edu.co}

\author[Valdivino V. Junior]{V. V. Junior}
\address[Valdivino V. Junior]{Institute of Mathematics and Statistics, Federal University of Goias, Campus Samambaia, 
CEP 74001-970, Goi\^ania, GO, Brazil}
\email{vvjunior@ufg.br}
\thanks{F\'abio Machado was supported by CNPq (303699/2018-3) and Fapesp (2017/10555-0). Alejandro Rold\'an was supported by Fapesp (2022/08948-2) and Universidad de Antioquia.}

\author[F\'abio P. Machado]{F. P. Machado}
\address[F\'abio P. Machado]{Statistics Department, Institute of Mathematics and Statistics, University of S\~ao Paulo, CEP 05508-090, S\~ao Paulo, SP, Brazil.}
\email{fmachado@ime.usp.br}

\author[Alejandro Rold\'an]{A. Rold\'an-Correa}
\address[Alejandro Rold\'an]{Instituto de Matem\'aticas, Universidad de Antioquia, Calle 67 no 53-108, Medellin, Colombia}
\email{alejandro.roldan@udea.edu.co}


\date{\today}

\begin{abstract}

Populations are often subject to catastrophes that cause mass removal of individuals. Many stochastic growth models have been considered to explain such dynamics. Among the results reported, it has been considered whether dispersion strategies, at times of catastrophes, increase the survival probability of the population. In this paper, we contrast dispersion strategies comparing mean extinction times of the population when extinction occurs almost surely. In particular, we consider populations subject to binomial catastrophes, that is, the population size is reduced according to a binomial law when a catastrophe occurs. Our findings delineate the optimal strategy (dispersion or non-dispersion) based on variations in model parameter values.
\end{abstract}

\maketitle

\section{Introduction}
\label{S: Introduction}
Several stochastic growth models have been considered to represent populations subject to catastrophes. When a catastrophe strikes, a random number of individuals are removed from the population. Survivors may remain together in the same colony (no dispersion) or disperse, making newly independent colonies. The interest in these models is to better understand quantities such as population survival probability, extinction time distribution, mean number of individuals removed, and the distribution of maximum population size, among others. 
The references \cite{AEL2007, B1986, BGR1982, CairnsPollett, KPR2016, KG2021, KG2022} pertain to population models where catastrophe survivors remain united in the same colony, while the models examined in \cite{JMR2016, JMR2020, MRV2018, MRS2015, S2014} investigate population dynamics with survivors dispersing to establish new colonies elsewhere.
In these papers, different types of catastrophes and different dispersion schemes are considered to analyze whether some of these schemes combined increase population viability. In biological context, it is known that dispersion holds a central role for both the dynamics and evolution of spatially structured populations. While it could save a small population from local extinction, it also could increase global extinction risk if observed in a very high level; refer to Ronce~\cite{R2007} for additional details.

The models analyzed in \cite{JMR2016,JMR2020, MRV2018,MRS2015, S2014} aim to establish, between dispersion and no dispersion, which is the best strategy, based on the survival probability of the population. When the survival probability is zero for both strategies, we need to go one step further considering the expected extinction time as this quantity is of particular importance to estimate the \lq\lq minimum viable population size\rq\rq\ to guarantee survival for a certain time, as considered in Brockwell~\cite{B1986}. 
For models of a single colony (no dispersion), one can find closed-form formulas for the mean extinction times for different types of catastrophes (see \cite{AEL2007, B1986, CairnsPollett}). For models with dispersion, analogous approach was considered for \textit{geometric catastrophes} (see~\cite{JMR2022}). 
Geometric catastrophes assume that the batch of removed individuals, when a catastrophe strikes, follows a geometric law; that is, the individuals are exposed to the catastrophic effect sequentially and the decline in the population stops at the first individual who
survives, or when the whole population in the colony becomes extinct.

Here we work with binomial catastrophes in models with dispersion. In binomial catastrophes the individuals of a colony are exposed to the catastrophic effect simultaneously and every individual survives a catastrophe with same probability, independently of anything else. We are able to present closed-form formulas for the mean extinction times and make comparisons with models without dispersion. Our analysis involves comparisons, by numerical and analytical methods, with functions expressed as infinite products, also known as \textit{infinite  $q$-products}. They are part of the theory of \textit{$q$-series} (see \cite{Berndt}). Further instances and applications of geometric catastrophes are detailed in \cite{AEL2007, EG2007, TH2020, KBG2020,KG2021}, while examples and applications of binomial catastrophes are found in \cite{AEL2007, B1986, BGR1982, KPR2016, KG2020, KG2022}.

In conclusion, here we propose to evaluate which strategy is better when extinction occurs almost surely, considering the mean extinction times for populations subject to \textit{binomial catastrophes}, that is, in the case when a population is hit by a catastrophe, its size is reduced according to a binomial distribution. In Section 2 we present the non dispersion model proposed in Artalejo~\textit{et al.}~\cite{AEL2007} and the models with dispersion proposed in Junior~\textit{et al.}~\cite{JMR2016}. Besides, we reach new results for these models. In Section 3 we discuss dispersal schemes as strategies for increasing life expectancy. In Section 4 we prove the results presented in Sections 2 and 3. Finally, in Section~\ref{appendix}, a numerical algorithm is developed allowing us to make calculations and comparisons with the infinite products that appear in the article.

This type of study provides predictive insights into population dynamics, aiding in conservation strategies and risk assessment. By quantifying vulnerability, informing policy decisions, and refining models, this research line contributes to both scientific understanding and practical applications. 

\section{Models and Results}

\subsection{Binomial catastrophe}

Populations are frequently exposed to catastrophic events that cause massive elimination of their individuals, for example, habitat destruction, environmental disaster, epidemics, etc. A catastrophe can instantly wipe out the entire population or just a part of it. In order to model such events, it is assumed that when a population is hit by a catastrophe, its size is reduced according to some law of probability. 
For catastrophes that reach the individuals simultaneously and independently of everything else, the appropriate model assume a binomial probability law. That is, if at a catastrophe time  the size of the population is $i$, it is reduced to $j$ with probability 
\[ \mu_{ij} =  {i\choose j}p^j(1-p)^{i-j}, \hspace{1cm}   0\leq j \leq i,\]
where $p\in(0,1)$ is the probability that each individual survives the catastrophe. The form of $\mu_{ij}$ represents what is called \textit{binomial catastrophe}.


\subsection{Growth model without dispersion}
Artalejo \textit{et al.}~\cite{AEL2007} present a model for a population which sticks together in one colony, without dispersion.  That colony gives birth to new individuals at rate $\lambda>0$,  while binomial catastrophes happen at rate $\mu$.

The population size (number of individuals in the colony) at time $t$ is a continuous time Markov process $\left\{X(t):t\geq 0\right\}$ that we denote by $\mc$. With the intention of making the formulas more straightforward and simplify the analysis, we take $\mu=1$ and set $X(0)=1$. 

Artalejo \textit{et al.}~\cite{AEL2007} use the word \textit{extinction} to describe the event that $X(t) = 0$, for some $t>0$, for a process where state 0 is not an absorbing state. In fact the extinction time here is the first hitting time to the state 0,  
$$\tau_A:=\inf\{t>0:X(t)=0\}.$$ The probability of extinction of $\mc$ is denoted by $\psi_A=\mathbb{P}[\tau_A<\infty].$ Its complement, $1-\psi_A$, is called survival probability. Artalejo~\textit{et al.}~\cite{AEL2007} proved that $\psi_A=1$ (extinction occurs almost surely) for all $\lambda>0$ and $0<p<1.$ The next result establishes the mean time of extinction for $\mc$.

\begin{theorem}[Artalejo \textit{et al.}~\cite{AEL2007}]
\label{th:semdisptime}
For the process $\mc$,     
$$\mathbb{E}[\tau_A]=\frac{1}{\lambda}\left(\prod_{k=0}^{\infty}\left(1+\lambda p^k\right)-1 \right).
$$
\end{theorem}

\begin{obs}
The infinite product $\prod_{k=0}^{\infty} (1+ \lambda p^k)$ is convergent for all $|p|<1$ and $\lambda \in \mathbb{R}$. For series representations and other properties of this infinite product, see  \cite[Corollary 2.3]{Berndt} and  \cite[Theorem 10.10]{Charalambides}.
\end{obs}

\subsection{Growth models with dispersion and spatial restriction.}
Let $\mathbb{T}_d^+$ be an infinite rooted tree whose vertices
have degree $d+1$, except the root that has degree $d$. Let us define a process with dispersion on $\mathbb{T}_d^+$, starting from a single colony placed at the root of $\mathbb{T}_d^+$, with just one individual. The number of individuals in a colony grows following a Poisson process of rate $\lambda>0$. To each colony we associate an exponential time of mean 1 that indicates when the binomial catastrophe strikes a colony. 
Each one of the individuals that survived the catastrophe picks
randomly a neighbor vertex between the $d$ neighboring vertices furthest from the root to create new colonies. Among the survivors that go to the same vertex to create a new colony at it, only one succeeds, the others die. So in this case  when a catastrophe occurs in a colony, that colony is replaced by 0,1, ... or $ d $ colonies. Let us denote this process with by $\MCI$.
	
$\MCI$ is a continuous-time Markov process with state space $\{0,1,2,3, \dots\}^{\mathbb{T}^d}$. For each particular realization of this process, we say that it {\it survives} if for any instant of time there is at least one colony somewhere. Otherwise, we say that it {\it dies out }. We denoted by $\psi_d$, the probability of extinction of $\MCI$. Junior~\textit{et al.}~\cite[Theorem 2.8]{JMR2016} showed that $\psi_d<1$ if and only if $p>\frac{d}{d+(d-1)\lambda}$, showing that there is a phase transition with respect to the parameter $p$.
	
It is clear that when  $\psi_d<1$, the extinction mean time for the process $\MCI$ is infinite. In the next results,  we derive the extinction mean time when extinction occurs almost surely, when $d=2$ and $d=3$.

 \begin{theorem}\label{MCItime}
	Let $\tau_d$ the extinction time of the process $\MCI$. 
	\begin{itemize}
		\item[$(i)$] If $ p < \frac{2}{\lambda+2}$, then 
	$$\mathbb{E}[\tau_2]=\frac{(\lambda p+1)(\lambda p +2)}{\lambda p^2(\lambda +1)}\ln\left[\frac{(1-p)(\lambda p+2)}{(1-p)(\lambda p +2)-\lambda p^2(\lambda +1)}\right].$$
	If $ p = \frac{2}{\lambda+2}$, then 
	$\mathbb{E}[\tau_2]=\infty.$
	\item[$(ii)$] If   $p < \frac{3}{2\lambda+3}$, then
	$$\mathbb{E}[\tau_3]=
	\frac{2\lambda p+3}{2 g(\lambda,p)}\ln\left[\frac{3-3p-\lambda p + g(\lambda,p)}{3-3p-\lambda p - g(\lambda,p)}\right],$$
	where 
	\begin{equation}\label{functiong}
	g(\lambda,p)= \sqrt{\frac{\lambda^2p^3(\lambda+1)(6+\lambda p-3p)}{(\lambda p +3)(\lambda p+1)}}
	\end{equation}
	
	If   $p = \frac{3}{2\lambda+3} $, then $\mathbb{E}[\tau_3]=\infty.$
	
   \end{itemize}
\end{theorem}

\subsection{Growth model with dispersion but no spatial restrictions.} 
	Consider a population of individuals divided into separate colonies. Each colony begins with  
	an individual. The number of individuals in each colony increases independently according 
	to a Poisson process of rate $\lambda > 0 $.
	To each colony we associate an exponential time of mean 1 that indicates when the binomial catastrophe strikes a colony. Each individual that survived the catastrophe 
	begins a new colony independently of everything else.
	We denote this process by $C_*(\lambda,p)$ and consider it starting from a single colony with just one
	individual.\\

	For each particular realization of $C_*(\lambda,p)$, we say that it {\it survives} if for any instant of time there is at least one colony somewhere. Otherwise, we say that it {\it dies out.} 
	We denoted by $\psi_*$, the probability of extinction of $C_*(\lambda,p).$ Junior {\textit et al.} \cite[Theorem 2.3]{JMR2016} showed that $\psi_*<1$ if and only if $p>\displaystyle\frac{1}{\lambda+1}$. 
	\\
	
	It is clear that when $\psi_*<1$, the extinction mean time of $C_*(\lambda,p)$ is infinite.
	The following theorem establishes the mean time of extinction for $C_*(\lambda,p)$ when $\psi_*=1.$  
	\begin{theorem}\label{th:comdisptime}
		Let $\tau_*$ the extinction time of the process $C_*(\lambda,p)$. Then      
		$$\mathbb{E}[\tau_*]=\left\{
		\begin{array}{ccl}
		1-\displaystyle\frac{\lambda+1}{\lambda}\ln\left[1-\frac{\lambda p}{1-p}\right]&,& \text{if }   p<\displaystyle\frac{1}{\lambda+1};\\ \\
		\infty&, & \text{if }  p=\displaystyle\frac{1}{\lambda+1}.
		\end{array}\right.
		$$
\end{theorem}

\subsection{Connections with branching processes}
The models $\MCI$ and $C_*(\lambda,p)$ are special versions of branching processes.  Next, we present an alternative description of these models.\\

Let $N_.$ be Poisson process with rate $\lambda$, and $N_0 = 1$. Let $J$ be an exponential random variable with rate 1, independent of $N_.$. Consider a population of size $N_J$ that undergoes a catastrophe: each element of it survives with probability $p$, independently of the rest. After the catastrophe, the surviving population size is then $Z = \text{Bin}(N_J, p)$ (a binomial random variable). If this number is zero, set $B = 0.$ Otherwise, label each element of the surviving population with a type $1, \ldots, d,$ independently of the others, and let $B$ be the number of distinct resulting types.

Next, consider a continuous time branching process with rate 1 and offspring distribution $B$. The resulting model is $\MCI$. The limit corresponding to $d\to\infty$ (number of types is the same as the size of the surviving population) is  $C_*(\lambda,p)$ and we refer to it informally as the $d = \infty$ case. In particular, for $C_*(\lambda,p)$, the offspring distribution conditioned on $Z$ is $\delta_Z$. As for $\MCI$, $d <\infty$, the offspring distribution conditioned on $Z = n$ has equal to $k = 0, 1,\ldots, d$ with probability $p_{n,k}$. In light of the above construction, $p_{0,\cdot} = \delta_0$, and for $n \geq 1$, and $k = 1, \ldots, d\wedge n$, a combinatorial calculation gives
\[p_{n,k}=\frac{1}{d^n}{d\choose k} \sum_{r_1,\ldots,r_k\geq1, r_1+\cdots+r_k=n}\frac{n!}{r_1!\cdots r_k!}\]
(for all other values $p_{n,k} = 0$). The formula above represents the proportion of ways to label $n$ items resulting in exactly $k$ distinct labels from the set $\{1,\ldots, d\}$. The expression for $p_{n,k}$ is manageable when $d = 2, 3$ and gets more complicated as $d$ gets larger. Note however that as $d \to\infty $, $p_{n,\cdot} \to \delta_n $, which is exactly observed for $d = \infty.$

\section{Discussion}\label{Section:Discussion}

In the presence of binomial catastrophes, dispersion is a good strategy to increase the probability of survival of the population. When there is no dispersion the probability of survival is always zero, see Artalejo~\textit{et al.}~\cite[Theorem 3.1]{AEL2007}). However, when there is dispersion the probability of survival can be positive depending on the parameters $\lambda$ and $p$, see Junior~\textit{et al.}~\cite[Theorems 2.3 and 2.8]{JMR2016} for details. An interesting question is to determine whether,
when the processes $\mc$, $\MCI$ and $\MCsr$  dies out almost surely, dispersion is an advantage or not for extend the population's life span. The answer is not trivial. Note that the growth and catastrophe rates are $n\lambda$ and $n$, respectively, whenever there are $n$ colonies in the whole population. Moreover, a catastrophe is more likely to wipe out a smaller colony than a larger one. On the other hand, multiple colonies provide multiple chances for survival (because the catastrophe only affects the colony where it occurs) and this may be a critical advantage of the processes $\MCI$ and $\MCsr$ over the process  $\mc$. Also note that in the $\MCI$ process, due to space constraints, during dispersion, some individuals may end up at the same spatial location. In this case, all but one individuals die. As a consequence, there is a dispute: On one hand, dispersion creates independent populations and thus contributes to survival. On the other hand, dispersion leads to death due to competition for space. 

Next result provides a comparison of the average times until extinction between processes $\mc$ and $\MCId$, under the condition that extinction happens almost surely in both processes.

\begin{prop}\label{indepte2} Assume  $p < \frac{2}{2+\lambda}.$ Then $\mathbb{E}[\tau_A]<\mathbb{E}[\tau_2]$ if and only if	
	\begin{equation}\label{eq:indepte2}
	\prod_{k=0}^{\infty}\left(1+\lambda p^k\right) <1+\frac{(\lambda p+1)(\lambda p +2)}{ p^2(\lambda +1)}\ln\left[\frac{(1-p)(\lambda p+2)}{(1-p)(\lambda p +2)-\lambda p^2(\lambda +1)}\right].
	\end{equation}
Moreover, $\mathbb{E}[\tau_A]= \mathbb{E}[\tau_2]$ if and only if we have an equality in (\ref{eq:indepte2}).
\end{prop}

Proposition~\ref{indepte2} is a consequence of Theorems~\ref{th:semdisptime} and \ref{MCItime}$(i)$. 
In Section~\ref{appendix} we develop a numerical algorithm that allows the computation and comparison of the function $f(p, \lambda)=\prod_{k=0}^{\infty} (1+ \lambda p^k).$ 
In particular, we can verify whether and where, in terms of the parametric space, inequality (\ref{eq:indepte2}) holds. From Proposition~\ref{indepte2} we can conclude that dispersion is a better strategy compared to
non-dispersion, when the parameters $(\lambda, p)$ fall in the gray region of Figure~\ref{fig:indepte2}. The opposite (non-dispersion is a better strategy than  dispersion) holds in the yellow region. 
Furthermore, Junior~\textit{et al}~\cite[Theorem 2.8]{JMR2016} show that the extinction probabilities in the white region of Figure~\ref{fig:indepte2} satisfies $\psi_2<1=\psi_A.$ In conclusion, still in the white region,  dispersion  is a better strategy than non-dispersion.

\begin{figure}[ht]
	\begin{tabular}{ccc}
		$\lambda$ & \parbox[c]{9cm}{\includegraphics[trim={0cm 0cm 0cm 0cm}, clip, width=9cm]{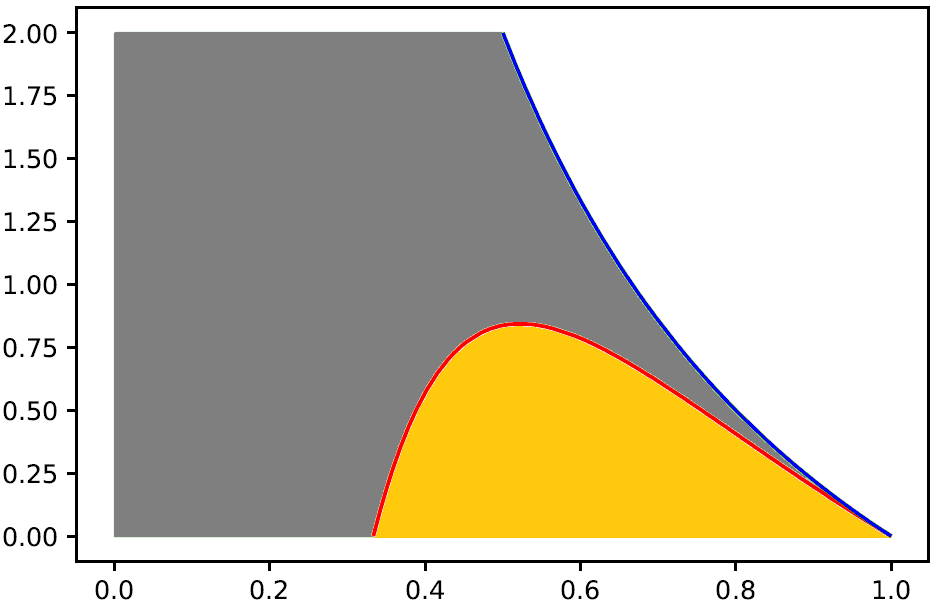}} & 
		\begin{tabular}{l}
			\textbf{\textcolor{red}{------}} Equality in (\ref{eq:indepte2})\\ \\
			\textbf{\textcolor{blue}{------}}  $p=\frac{2}{\lambda +2}$
		\end{tabular} \\
		& $p$
	\end{tabular}
	\caption{In the gray region, $\mathbb{E}[\tau_A]<\mathbb{E}[\tau_2]$. In the yellow region, $\mathbb{E}[\tau_A]>\mathbb{E}[\tau_2]$.}
	\label{fig:indepte2}
\end{figure}

\begin{exa} Both processes, $C(1/2,p)$ and $C_2(1/2,p)$, die out if and only if $p\leq4/5$. 	
In this case, considering (\ref{eq:indepte2}), we obtain $p_l\approx 0.38 $ and $p_u\approx 0.75$ such that: 
\begin{itemize}
	\item If $p\in(p_l,p_u)$, then  $\mathbb{E}[\tau_2]<\mathbb{E}[\tau_A]<\infty$.
	\item If $p=p_l$ or $p=p_u$, then $\mathbb{E}[\tau_A]=\mathbb{E}[\tau_2]<\infty$.
	\item If $p\in(0,p_l)\cup(p_u,4/5)$, then $\mathbb{E}[\tau_A]<\mathbb{E}[\tau_2]<\infty$. 
	\item If $p\geq4/5$, then $\mathbb{E}[\tau_A]<\mathbb{E}[\tau_2]=\infty$. 
\end{itemize}  
\end{exa}

The following result establishes a comparison between the mean extinction times for the processes $\mc$ and $\MCIt$, when extinction occurs almost surely in both processes.

\begin{prop}\label{indepte3} Assume $p<\frac{3}{2\lambda+3}$. Then $\mathbb{E}[\tau_A]<\mathbb{E}[\tau_3]$ in and only if  
\begin{equation}\label{eq:indepte3}	
	\prod_{k=0}^{\infty}\left(1+\lambda p^k\right)<1+ \frac{\lambda(2\lambda p+3)}{2 g(\lambda,p)}\ln\left[\frac{3-3p-\lambda p + g(\lambda,p)}{3-3p-\lambda p - g(\lambda,p)}\right],
\end{equation}
where $g(\lambda,p)$ is given by (\ref{functiong}).
 Moreover,   $\mathbb{E}[\tau_A]=\mathbb{E}[\tau_3]$ if and only if we have an equality in (\ref{eq:indepte3}).
\end{prop}

Proposition~\ref{indepte3} is a consequence of Theorems~\ref{th:semdisptime} and \ref{MCItime}$(ii)$. From Proposition~\ref{indepte3} we can conclude that  dispersion is a better strategy compared to non-dispersion, when the parameters $(\lambda, p)$ fall in the gray region of Figure~\ref{fig:indepte3}. The opposite (non-dispersion is a better strategy than independent dispersion) holds in the yellow region. Furthermore, Junior~\textit{et al}~\cite[Theorem 2.8]{JMR2016} show that the extinction probabilities in the white region of Figure~\ref{fig:indepte3} satisfies $\psi_3<1=\psi_A.$ Thus, in the white region,  dispersion  is a better strategy than non-dispersion.

\begin{figure}[ht]
	\begin{tabular}{ccc}
		$\lambda$ & \parbox[c]{9cm}{\includegraphics[trim={0cm 0cm 0cm 0cm}, clip, width=9cm]{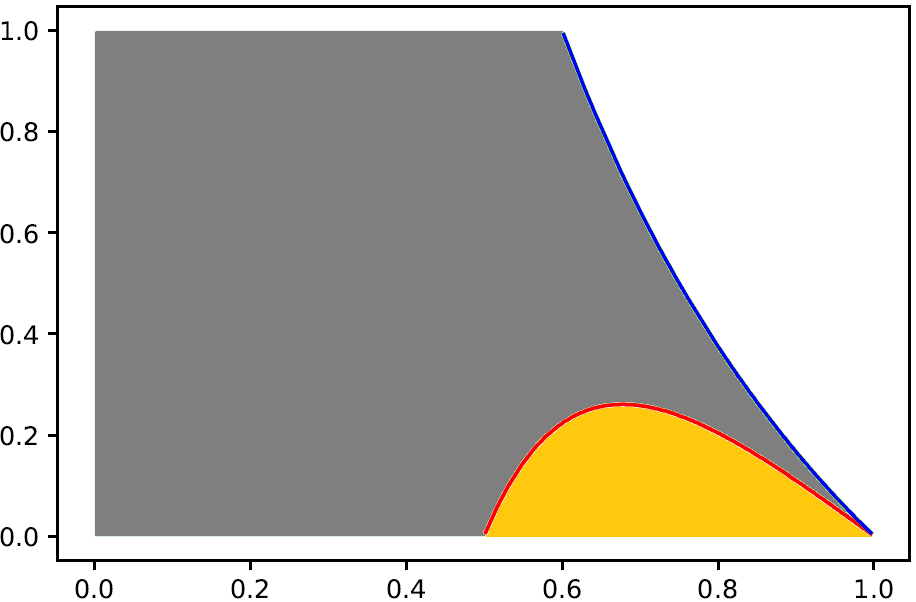}} & 
		\begin{tabular}{l}
			\textbf{\textcolor{red}{------}} Equality in (\ref{eq:indepte3})\\ \\ 
			\textbf{\textcolor{blue}{------}}  $p=\frac{3}{2\lambda +3}$
		\end{tabular} \\
		& $p$
	\end{tabular}
	\caption{In the gray region, $\mathbb{E}[\tau_A]<\mathbb{E}[\tau_3]$. In the yellow region, $\mathbb{E}[\tau_A]>\mathbb{E}[\tau_3]$.}
	\label{fig:indepte3}
\end{figure}

\begin{exa} Both processes, $C(1/5,p)$ and $C_3(1/5,p)$, die out if and only if $p\leq15/17$. 	
	In this case, considering (\ref{eq:indepte3}), we obtain $p_l\approx 0.58$ and $p_u\approx0.80$ such that: 
	\begin{itemize}
		\item If $p\in(p_l,p_u)$, then  $\mathbb{E}[\tau_3]<\mathbb{E}[\tau_A]<\infty$.
		\item If $p=p_l$ or $p=p_u$, then $\mathbb{E}[\tau_A]=\mathbb{E}[\tau_3]<\infty$.
		\item If $p\in(0,p_l)\cup(p_u,15/17)$, then $\mathbb{E}[\tau_A]<\mathbb{E}[\tau_3]<\infty$. 
		\item If $p\geq15/17$, then $\mathbb{E}[\tau_A]<\mathbb{E}[\tau_3]=\infty$. 
	\end{itemize}  
\end{exa}	

The following result establishes that  the mean extinction time for the process without dispersion, $\mc$, is less than for the  process with dispersion and no spatial restriction, $\MCsr$, when extinction occurs almost surely in both processes.

\begin{prop}\label{disp}
	If $p<\frac{1}{\lambda+1}$,  then
	$\mathbb{E}[\tau_A]< \mathbb{E}[\tau_*].$
\end{prop}

Proposition~\ref{disp} leads us to the conclusion that, in the absence of spatial constraints and under binomial catastrophes, dispersion is a more effective strategy than non-dispersion in extending the population's lifespan.


\section{Proofs}

\begin{lem}[Lemma 4.1 in Junior~\textit{et al.}~\cite{JMR2020}]\label{lemaux}
	Let $(Y_t)_{t\geq 0}$ a continuous time branching process, where each particle survives an exponential time of rate 1 and right before death produces a random number of particles  with probability generating function
	$$f(s)=\sum_{k=0}^{\infty}p_ks^k.$$
	Suppose that $Y_0=1$ and $f'(1)\leq 1$.  Let $\tau=\inf\{t>0:Y_t=0\}$,  the extinction time of the process $(Y_t)_{t\geq 0}$. 
	\begin{enumerate}
		\item[$(i)$] If $p_2\neq 0$ and $p_k= 0$ for $k\geq 3$, then
		$$\mathbb{E}[\tau]=
		\left\{\begin{array}{cl}
		\displaystyle\frac{1}{p_2}\ln\left(\frac{p_0}{p_0-p_2}\right) & \text{, if } f'(1)<1,\\ \\
		\infty & \text{, if } f'(1)=1.
		\end{array}
		\right.$$
		\item[$(ii)$] If $p_3\neq 0$ and $p_k= 0$ for $k\geq 4$, then
		$$\mathbb{E}[\tau]=
		\left\{\begin{array}{cl}
		\displaystyle\frac{1}{\sqrt{4p_0p_3+(p_2+p_3)^2}}\ln\left[\frac{2p_0-p_2-p_3+\sqrt{4p_0p_3+(p_2+p_3)^2}}{2p_0-p_2-p_3-\sqrt{4p_0p_3+(p_2+p_3)^2}}\right] & \text{, if } f'(1)<1,\\ \\
		\infty & \text{, if } f'(1)=1.
		\end{array}
		\right.$$
		\item[$(iii)$] If $p_0=\beta$ and $p_n=\alpha c^n$ for $n\geq 1$, where $\alpha, \beta$ and $c$ are positive constants, then
			$$\mathbb{E}[\tau]=
			\left\{\begin{array}{cl}
			1-\displaystyle\frac{1-\beta}{c}\ln\left[1-\frac{c}{\beta}\right] & \text{, if } f'(1)<1,\\ \\
			\infty & \text{, if } f'(1)=1.
			\end{array}
			\right.$$\\
		\end{enumerate}
\end{lem}

In order to prove Theorems~\ref{MCItime} and \ref{th:comdisptime}, observe that  the probability distribution of the number of survivals right after the catastrophe (but before the dispersion) is given by
\[\mathbb{P}(N=0) =\beta, \,\mathbb{P}(N = n) = \alpha c^n, n =1,2,\ldots,\] where
\begin{equation}\label{beta_lambda_c}
\beta = \frac{1-p}{\lambda p + 1}, \ \alpha = \frac{\lambda + 1}{\lambda ( \lambda p +1)} \hbox{ and } \ c = \frac{\lambda p}{\lambda p + 1}.
\end{equation} 
For details see Machado~\textit{et al}~\cite[Equation (4.1)]{JMR2016}.

\begin{proof}[Proof of Theorem~\ref{MCItime}] 
Let $Z_t$ be the number of colonies at time $t$ in the model $\MCI$. Observe that $Z_t$ is a continuous-time branching process with  $Z_0=1$. Each particle (colony) in $Z_t$ survives an exponential time of rate 1 and right before death produces $k\leq d$ particles (colonies are created right after a catastrophe) with probability $p_k$ given by

	$$p_k=\left\{
\begin{array}{cl}
\beta & \text{, if }  k=0;\\ 
\alpha\dbinom{d}{k}\displaystyle\sum_{n=k}^{\infty} T(n,k)\left(\frac
{c}{d}\right)^n   & \text{, if } 1\leq k<d; \\ 
1- \displaystyle\sum_{j=0}^{d-1}p_j & \text{, if } k= d; 
\end{array}\right.
$$
where $T(n, k)$ denote the number of surjective functions $f : A \to B,$ with $|A| = n$
and $|B| = k.$\\
		
Moreover, $\tau_d^i=\inf\{t>0:Z_t=0\}$.	 \\	
	
$\bullet$ For $d=2$, we have that  
	\[
p_0 = \beta, \ p_1 = \frac{2\alpha c}{2-c} \hbox{ and } p_2 = 1- \beta - \frac{2\alpha c}{2-c}.
\]
Furthermore, the condition $p< \frac{2}{2+\lambda}$  is equivalent to $p_1+2p_2<1$.  Thus, from Lemma~\ref{lemaux}$(i)$, we have that  	
	$$\begin{array}{lll}
	\mathbb{E}[\tau_2]&=&\displaystyle\frac{1}{p_2}\ln\left(\frac{p_0}{p_0-p_2}\right)\\ \\
	&=&\displaystyle\frac{(\lambda p+1)(\lambda p +2)}{\lambda p^2(\lambda +1)}\ln\left[\frac{(1-p)(\lambda p+2)}{(1-p)(\lambda p +2)-\lambda p^2(\lambda +1)}\right],
	\end{array}$$
where the last line has been obtained using (\ref{beta_lambda_c}).\\

When $p= \frac{2}{2+\lambda}$, we have that  $p_1+2p_2=1$.  Thus, from Lemma~\ref{lemaux}$(i)$, it follows that  $	\mathbb{E}[\tau_2]=\infty$.\\

$\bullet$ For $d=3$, we have that 
	\[
p_0 = \beta,  p_1 = \frac{3\alpha c}{3-c}, p_2= \frac{6\alpha c^2}{(3-2c)(3-c)}  \hbox{ and } p_3 = 1- \beta- \frac{3\alpha c}{3-c}-\frac{6\alpha c^2}{(3-2c)(3-c)}.
\]
Furthermore, the condition $p< \frac{3}{2\lambda+3}$  is equivalent to $p_1+2p_2+3p_3<1$. Thus, from Lemma~\ref{lemaux}$(ii)$, we have that  		
	  	
$$\begin{array}{lll}
\mathbb{E}[\tau_3]&=&\displaystyle\frac{1}{\sqrt{4p_0p_3+(p_2+p_3)^2}}\ln\left[\frac{2p_0-p_2-p_3+\sqrt{4p_0p_3+(p_2+p_3)^2}}{2p_0-p_2-p_3-\sqrt{4p_0p_3+(p_2+p_3)^2}}\right]\\\\
&=&
\displaystyle\frac{2\lambda p+3}{2 g(\lambda,p)}\ln\left[\frac{3-3p-\lambda p + g(\lambda,p)}{3-3p-\lambda p - g(\lambda,p)}\right],
\end{array}$$
where the last line has been obtained using (\ref{beta_lambda_c}) and $g(\lambda,p)$ is given by (\ref{functiong}).\\

When $p= \frac{3}{2\lambda+3}$, we have that $p_1+2p_2+3p_3=1$.  Thus, from Lemma~\ref{lemaux}$(ii)$, it follows that  $	\mathbb{E}[\tau_3]=\infty$.
\end{proof}

\begin{proof}[Proof of Theorem~\ref{th:comdisptime}] Analogously to the proof of Theorem~\ref{MCItime}. In this case, 
$p_0 =\beta,$ and $p_n = \alpha c^n, n =1,2,\ldots.$ 
\end{proof}

\begin{proof}[Proof of Proposition~\ref{disp}]
Assume that $p<\frac{1}{\lambda+1}$ (or equivalently $\lambda p<1-p)$. From Theorems~\ref{th:semdisptime} and \ref{th:comdisptime} we have that  
$\mathbb{E}[\tau_A]<\mathbb{E}[\tau_*]$ if and only if   
\begin{equation}\label{eq:disp}
	\prod_{k=1}^{\infty}\left(1+\lambda p^k\right) < 1-\displaystyle\ln\left(1-\frac{\lambda p}{1-p}\right).
\end{equation}	
To show that inequality (\ref{eq:disp}) holds, note that the series 
$$\sum_{n=1}^{\infty}\sum_{k=1}^{\infty}\frac{(-1)^{n+1}}{n}\lambda^n p^{kn}=\sum_{n=1}^{\infty}\frac{(-1)^{n+1}}{n} \frac{\lambda^np^n}{1-p^n},$$
 converges  absolutely if $\lambda p<1-p$ (use the root test). Thus, using the Taylor expansion (see \cite[Chapter 9]{Bartle2}) of the function $\ln(1+x)$  and  Fubini's Theorem (see \cite[Chapter 10]{Bartle1}), we have 
\begin{eqnarray*}
	\ln\left[\prod_{k=1}^{\infty}\left(1+\lambda p^k\right)\right]&=&\sum_{k=1}^{\infty}\ln(1+\lambda p^k) \nonumber\\
	&=&\sum_{k=1}^{\infty}\sum_{n=1}^{\infty}\frac{(-1)^{n+1}}{n}\lambda^n p^{kn} \nonumber\\
	&=&\sum_{n=1}^{\infty}\sum_{k=1}^{\infty}\frac{(-1)^{n+1}}{n}\lambda^n p^{kn} \nonumber \\
	&=&\sum_{n=1}^{\infty}\frac{(-1)^{n+1}}{n} \frac{\lambda^np^n}{1-p^n}.
\end{eqnarray*}	
Let $a_n=\frac{(-1)^{n+1}}{n}\frac{\lambda^np^n}{1-p^n}.$ Observe that for $\lambda p<1-p,$
\begin{eqnarray*}
a_{2n}+a_{2n+1}&=& -\frac{\lambda^{2n}p^{2n}}{2n(1-p^{2n})} + \frac{\lambda^{2n+1}p^{2n+1}}{(2n+1)(1-p^{2n+1})} \nonumber\\
&=& -\frac{\lambda^{2n}p^{2n}}{2n(2n+1)}\left[\frac{2n+1}{1-p^{2n}}-\frac{2n\lambda p}{1-p^{2n+1}}\right]  \nonumber\\
&<& -\frac{\lambda^{2n}p^{2n}}{2n(2n+1)}\left[\frac{2n+1}{1-p^{2n}}-\frac{2n(1-p)}{1-p^{2n+1}}\right]  \nonumber\\
&=& -\frac{\lambda^{2n}p^{2n}}{2n(2n+1)}\left[\frac{2n[p^{2n}(1-p)+p(1-p^{2n})]}{(1-p^{2n})(1-p^{2n+1})}+\frac{1}{1-p^{2n}}\right]  \nonumber\\
&<&0.
\end{eqnarray*}	
Thus, $$\ln\left[\prod_{k=1}^{\infty}\left(1+\lambda p^k\right)\right]=a_1+\sum_{n=1}^{\infty}(a_{2n}+a_{2n+1})<a_1=\frac{\lambda p}{1-p}.$$
Therefore, using the Taylor expansions of the functions $e^x$ and $\ln(1-x)$, we have that
\begin{eqnarray*}
\prod_{k=1}^{\infty}\left(1+\lambda p^k\right)&\leq& \exp\left(\frac{\lambda p}{1-p}\right)\\
&=&1+\sum_{n=1}^{\infty}\frac{1}{n!}\left(\frac{\lambda p}{1-p}\right)^n \\
&<&1+\sum_{n=1}^{\infty}\frac{1}{n}\left(\frac{\lambda p}{1-p}\right)^n\\
&=&1-\ln\left(1-\frac{\lambda p}{1-p}\right).
\end{eqnarray*}	
\end{proof}

\section{Numerical Analysis}\label{appendix}

This section presents the development of a numerical method to identify the regions within the parametric space of $p\times\lambda$ where Inequalities (\ref{eq:indepte2}) and (\ref{eq:indepte3}) hold, as well as the regions where they do not.

Let $$f(p, \lambda)=\prod_{k=0}^{\infty} (1+ \lambda p^k).$$
Let $g_1$, $g_2$, $h_1$ and $h_2$,  functions of $p$ and $\lambda$ such that  Inequalities (\ref{eq:indepte2}) and (\ref{eq:indepte3}) correspond to $f< g_1$ and $f< g_2$,  restricted to $h_1>0$ and $h_2>0$, respectively. 
Note that 
$$h_1(p,\lambda)= 2-p(\lambda+2)$$
$$h_2(p,\lambda)= 3-p(2\lambda+3)$$
In order to calculate and compare the function $f$, we use the lower and upper bounds given by the following lemma.

\begin{lem}\label{lema}
	If $p< \frac{a}{b \lambda +a}$.
	Then, for all $M\in\mathbb{N},$
	$$ \prod_{k=0}^{M} (1+ \lambda p^k) \leq f(p, \lambda) \leq
	\exp\left(\frac{a}{b}p^{M}\right) \prod_{k=0}^{M} (1+ \lambda p^k).  $$
\end{lem}

\begin{proof} The first inequality holds since $1+\lambda p^k\geq 1$ for all $k\geq 1$. To prove the second inequality we observe that as
    $p< \frac{a}{a+b \lambda}$, then $$\lambda <\frac{a}{b}\frac{(1-p)}{p}.$$
	
	Thus,  using $(1+ x)\leq e^x $ for all $x\in\mathbb{R},$ 

	we have that 
	$$\begin{array}{lll}
		\displaystyle\prod_{k=0}^{\infty} (1+ \lambda p^k)&\leq& \left[\displaystyle\prod_{k=0}^{M} (1+ \lambda p^k)\right] \displaystyle\exp\left(\sum_{k=M+1}^{\infty} \lambda p^{k} \right) \\ \\
		&=&\left[\displaystyle\prod_{k=0}^{M} (1+ \lambda p^k)\right] \displaystyle\exp\left(\frac{\lambda p^{M+1}}{1-p} \right) \\ \\
		&<&\left[\displaystyle\prod_{k=0}^{M} (1+ \lambda p^k)\right] \displaystyle\exp\left(\frac{a}{b}p^{M}\right).
	\end{array}   $$
\end{proof}

Now we consider the following task. Given particular values of $p$ and $\lambda$, determine if $f$ is lower than or greater than $g$. We do this by recursion on $M$: if the upper bound of $f$ is below $g$ then $f$ is lower than $g$,  if the lower bound of $f$ is above $g$ then $f$ is greater than $g$, in other case we try again with a bigger value of $M$. Notice that the upper and lower bounds of $f$ tend to $f$ when $M$ tends to infinity.\\

\section{Acknowledgments} 
The authors are grateful to the anonymous referees for their valuable comments and suggestions that contributed to enhancing the quality of the paper.


\begin{thebibliography}{99}

\bibitem{AEL2007} {J.R.Artalejo, A.Economou and M.J.Lopez-Herrero.} Evaluating growth measures in an immigration process subject to binomial and geometric catastrophes. \textit{Mathematical Biosciences and Engineering} \textbf{4}, (4), 573 - 594 (2007).

\bibitem{Bartle1}{R.G. Bartle.} The elements of integration and lebesgue measure. John Wiley (1995).

\bibitem{Bartle2}{R.G. Bartle and D. R. Sherbert.} Introduction to real analysis. Third Edition. John Wiley (1999).

\bibitem{Berndt}{B. C. Berndt.} What is a q-series?, in Ramanujan Rediscovered: Proceedings of a Conference on Elliptic Functions, Partitions, and q-Series in memory of K. Venkatachaliengar: Bangalore, 1--5 June, 2009, N. D. Baruah, B. C. Berndt, S. Cooper, T. Huber, and M. J. Schlosser, eds., Ramanujan Mathematical Society, Mysore, 2010, pp. 31--51.


\bibitem{B1986}{P.J.Brockwell}. The Extinction Time of a General Birth and Death Process with Catastrophes.\textit{Journal of Applied Probability},  \textbf{23} (4), 851-858 (1986). 

\bibitem{BGR1982}{P.J.Brockwell, J.Gani and S.I.Resnick.} Birth, immigration and catastrophe processes.
\textit{Adv. Appl. Prob}. \textbf{14}, 709-731 (1982).

%
\bibitem{CairnsPollett}{B.Cairns and P.K. Pollet.} Extinction Times for a General Birth, Death and Catastrophe Process. \textit{Journal of Applied Probability} \textbf{41} (4), 1211-–1218 (2004).

%
\bibitem{Charalambides}{C. A. Charalambides.}
 Enumerative Combinatorics. Chapman and Hall, Boca Raton, (2002).

\bibitem{EG2007} A. Economou and A. Gomez-Corral. The Batch Markovian Arrival Process Subject to Renewal Generated Geometric Catastrophes. \textit{Stochastic Models}, \textbf{23} (2), 211-233, (2007).

\bibitem{TH2020} Thierry Huillet. On random population growth punctuated by geometric catastrophic events. \textit{Contemporary Mathematics}, \textbf{1}, (5), pp.469 (2020).
%
\bibitem{JMR2016} {V.V. Junior,  F.P.Machado and A. Rold\'an-Correa.} 
Dispersion as a Survival Strategy.
\textit{Journal of Statistical Physics} 
\textbf{164} (4), 937 - 951 (2016). 

\bibitem{JMR2020}{V.V. Junior, F.P.Machado and A. Rold\'an-Correa.} Evaluating dispersion strategies in growth models subject to geometric catastrophes. \textit{Journal of Statistical Physics} \textbf{183}, 30 (2021).

\bibitem{JMR2022}{V.V. Junior, F.P.Machado and A. Rold\'an-Correa.} Extinction time in growth models subject to geometric catastrophes. \textit{Journal of Statistical Mechanics: Theory and Experiment}  \textbf{2023} (4), 043501, (2023).

\bibitem{KPR2016} S.Kapodistria, T. Phung-Duc and J. Resing. Linear birth/immigration-death process with binomial catastrophes. \textit{Probability in the Engineering and Informational Sciences} \textbf{30} (1), 79-111 (2016).
%
\bibitem{KBG2020} N.  Kumar, F. P.  Barbhuiya, U. C.  Gupta. Analysis of a geometric catastrophe model with discrete-time batch renewal arrival process. \textit{RAIRO-Oper. Res.} \textbf{54} (5), 1249-1268 (2020).

\bibitem{KG2020} N. Kumar and U.C. Gupta. Analysis of batch Bernoulli process subject to discrete-time renewal generated binomial catastrophes. \textit{Ann Oper Res}, \textbf{287}, 257-283 (2020).

\bibitem{KG2021} N. Kumar and U.C. Gupta.  Analysis of a population model with batch Markovian arrivals influenced by Markov arrival geometric catastrophes, \textit{Communications in Statistics - Theory and Methods}, \textbf{50}(13), 3137-3158 (2021)

\bibitem{KG2022} N. Kumar and U. C. Gupta. Markovian Arrival Process Subject to Renewal Generated Binomial Catastrophes, \textit{Methodology and Computing in Applied Probability}, Springer, \textbf{24}(4), 2287-2312 (2022).

\bibitem{MRV2018}{F.P.Machado, A. Rold\'an-Correa and V.V. Junior. } Colonization and Collapse on Homogeneous Trees.\textit{Journal of Statistical Physics} \textbf{173}, 1386-1407 (2018).


\bibitem{MRS2015}{F.P.Machado, A. Rold\'an-Correa and R.Schinazi.} 
Colonization and Collapse. \textit{ALEA-Latin American Journal of Probability and Mathematical Statistics} \textbf{14}, 719-731 (2017). 

\bibitem{R2007}{ O. Ronce.} How does it feel to be like a rolling stone? Ten questions about dispersal evolution. \textit{Annu. Rev. Ecol. Evol. Syst.} \textbf{38}, 231–253,(2007).

\bibitem{S2014} {R.Schinazi.} 
Does random dispersion help survival?
\textit{Journal of Statistical Physics}, \textbf{159}, (1), 101-107 (2015).

\end{thebibliography}
\end{document}